\documentclass{amsart}
\usepackage[margin=1.2in]{geometry}

\title{Perspectives on CUR Decompositions}

\usepackage{hyperref}
\usepackage{hyperref}
\usepackage{tablefootnote}
\usepackage{graphicx}
\usepackage{subcaption}
\usepackage{amsthm}
\usepackage{dsfont}
\usepackage{amssymb}
\usepackage{enumerate}
\usepackage{graphicx}
\usepackage{float}
\usepackage{bbm}
\usepackage{amsmath}
\usepackage{comment}
\usepackage{hyperref}
\usepackage{listings}
\usepackage{color}
\usepackage{ulem}
\usepackage[dvipsnames]{xcolor}
\usepackage[lined,boxed,commentsnumbered]{algorithm2e}
\allowdisplaybreaks

\hypersetup{
	colorlinks,
	citecolor=blue,
	filecolor=blue,
	linkcolor=blue,
	urlcolor=blue,
	hyperfootnotes=false
}

\newtheorem{theorem}{Theorem}[section]
\newtheorem{proposition}[theorem]{Proposition}
\newtheorem{lemma}[theorem]{Lemma}

\newtheorem{corollary}[theorem]{Corollary}

\newtheorem{thmnum}{Theorem}

\theoremstyle{definition}
\newtheorem{example}[theorem]{Example}

\newtheorem{problem}{Problem}

\newtheorem{remark}[theorem]{Remark}

\newcommand{\R}{\mathbb{R}}

\newcommand{\K}{\mathbb{K}}

\newcommand{\C}{\mathbb{C}}
\newcommand{\sspan}{\textnormal{span}}

\newcommand{\rank}{{\rm rank\,}}

\newcommand{\argmin}{\text{argmin}}

\begin{document}

\author{Keaton Hamm}
\address{Department of Mathematics, University of Arizona, Tucson, AZ 85719 USA}
\email{hamm@math.arizona.edu}

\author{Longxiu Huang}
\address{Department of Mathematics, University of California, Los Angeles, CA 90095 USA}
\email{huangl3@math.ucla.edu}

%%%%%%%%%%%%%%%%%%%%%%%%%%%%%%%%%%%%%%%%%%%
%%%%%%%%%%%%%%%%%%%%%%%%%%%%%%%%%%%%%%%%%%%
%%%%%%    Abstract                 %%%%%%%%
%%%%%%%%%%%%%%%%%%%%%%%%%%%%%%%%%%%%%%%%%%%

\begin{abstract}
This note discusses an interesting matrix factorization called the CUR Decomposition.  We illustrate various viewpoints of this method by comparing and contrasting them in different situations.  Additionally, we offer a new characterization of CUR decompositions which synergizes these viewpoints and shows that they are indeed the same in the exact decomposition case.
\end{abstract}

\maketitle

%%%%%%%%%%%%%%%%%%%%%%%%%%%%%%%%%%%%%%%%%%%
%%%%%%%%%%%%%%%%%%%%%%%%%%%%%%%%%%%%%%%%%%%
%%%%%%    Introduction             %%%%%%%%
%%%%%%%%%%%%%%%%%%%%%%%%%%%%%%%%%%%%%%%%%%%
%%%%%%%%%%%%%%%%%%%%%%%%%%%%%%%%%%%%%%%%%%%

\section{Introduction}

%One of the most influential tools of Linear Algebra is that of factorizing matrices.  In fact, this broad idea made its way into a list of the top ten algorithms of the twentieth century \cite{Top10} despite not being an algorithm!  In applications, matrix factorizations  are useful because they can compress storage, speed up computations, and allow for fast solutions to optimization problems; however, they are also of great utility in theory as well since factorizations can be done in which one of the terms contains essential information about the matrix.  Perhaps the classic starting point is the observation that every $m\times n$ matrix $A$ with real or complex entries admits a \textit{singular value decomposition (SVD)} of the form
%\[ A = U\Sigma V^*,\]
%where $U$ and $V$ are orthogonal matrices and $\Sigma$ is a matrix with nonzero entries only along the main diagonal which are {\color{red}the square roots of} the eigenvalues of $AA^*$ (equivalently, $A^* A$).

%In this note, we present a matrix factorization called the CUR decomposition, whose origins in some form go back at least to Penrose \cite{Penrose56}.  However, with the exception of certain pockets of the randomized linear algebra and theoretical computer science literature, this decomposition seems to be relatively little known.  Our main result is to characterize such decompositions in a manner which synthesizes both of the standard viewpoints of the decomposition, as well as giving some surprising equivalent formulations.  

In many data analysis applications, two key tools are dimensionality reduction and compression, in which data obtained as vectors in a high dimensional Euclidean space are approximated in a basis or frame which spans a much lower dimensional space (reduction) or a sketch of the total data matrix is made and stored in memory (compression).  Without such steps as preconditioners to further analysis, many problems would be intractable.  However, one must balance the approximation method with the demand that any conclusions made from the approximated version of the data still be readily interpreted by domain experts. This task can be challenging, and many well-known methods (for instance PCA) allow for great approximation and compression of the data, but at the cost of inhibiting interpretation of the results using the underlying physics or application.

One way around this difficulty is to attempt to utilize the \textit{self-expressiveness} of the data, which is the notion that oftentimes data is better represented in terms of linear combinations of other data points rather than in some abstract basis.  In many applications data is self-expressive, and methods based on this assumption achieve rather good results in various machine learning tasks (as a particular example, we refer the reader to the Sparse Subspace Clustering algorithm of Elhamifar and Vidal \cite{SSC}).  The question then is: how may one use self-expressiveness to achieve dimensionality reduction? Mahoney and Drineas \cite{DMPNAS} argue for the representation of a given data matrix in terms of \textit{actual columns and rows} of the matrix itself.  The idea is that, rather than do something like PCA to transform the data into an eigenspace representation, one attempts to choose the most representative columns and rows which capture the essential information of the initial matrix.  Thus enters the \textit{CUR Decomposition}.

There are two distinct starting points in most of the literature revolving around the CUR decomposition (also called (pseudo)skeleton approximations \cite{DemanetWu,Goreinov}).  The first is as an exact matrix \textit{decomposition}, or \textit{factorization}: given an arbitrary, and possibly complicated matrix $A$, one may desire to decompose $A$ into the product of 2 or more factors, each of which is ``easier" to understand, store, or compute with than $A$ itself. The exact decomposition (see Theorem \ref{THM:CUR} for a formal statement) says that $A=CU^\dagger R$ if $\rank(U)=\rank(A)$, where $C$ and $R$ are column and row submatrices of $A$, respectively, and $U$ is their intersection.   The second starting point is to find a \textit{low-rank approximation} to a given matrix. This is along the lines of the work of Drineas, Kannan, Mahoney, and others \cite{DKMIII,DM05,DMM08} and stems from the considerations of interpretability above.
This is the typical vantage point of much of randomized linear algebra (e.g., \cite{tropp}).  In this setting, the CUR approximation of a matrix is typically given by $A\approx CC^\dagger AR^\dagger R$, where $CC^\dagger$ and $R^\dagger R$ are orthogonal projections onto the subspaces spanned by given columns and rows, respectively.  These perspectives are not typically addressed in the literature together (indeed, the first perspective is taken but rarely).  

This short note compares and contrasts these two viewpoints of CUR, showing that they are indeed the same in the exact decomposition case, but distinct in the approximation case. Moreover, we obtain a characterization of CUR decompositions which provides some interesting and unexpected equivalences.

\section{Main Result}

The following is the main result of this note; here $\K$ is either $\R$ or $\C$, and $A^\dagger$ is the Moore--Penrose pseudoinverse of $A$ (see Section \ref{SEC:Notation} for the full definition; briefly, if $A=U\Sigma V^*$ is the SVD of $A$, then $A^\dagger = V\Sigma^\dagger U$, where the diagonal of $\Sigma^\dagger$ contains the reciprocal of the elements of $\Sigma$).    

\begin{thmnum}[Theorem  \ref{THM:Characterization} and Proposition \ref{PROP:Udagger}]
Let $C=A(:,J)$ and $R=A(I,:)$ be column and row submatrices of $A\in\K^{m\times n}$ (possibly with repeated columns/rows), and let $U=A(I,J)$ be their intersection.  Then the following are equivalent:
\begin{enumerate}[(i)]
    \item $\rank(U)=\rank(A)$
    \item $A=CU^\dagger R$
    \item $A = CC^\dagger AR^\dagger R$
    \item $A^\dagger = R^\dagger UC^\dagger$
    \item $\rank(C)=\rank(R)=\rank(A)$.
\end{enumerate}
Moreover, if any of the equivalent conditions above hold, then $U^\dagger = C^\dagger AR^\dagger$.
\end{thmnum}
This theorem reconciles the exact CUR decomposition $A=CU^\dagger R$ with the CUR approximation given by $A\approx CC^\dagger AR^\dagger R$.  The latter is the best CUR approximation in a sense that is made precise in the sequel.  The moreover part (and the equivalence of (\ref{ITEM:Adagger})) is interesting in its own right because it is not generally the case that $(AB)^\dagger = B^\dagger A^\dagger.$

The rest of the paper develops as follows: Section \ref{SEC:Notation} establishes the notation used throughout the sequel; Section \ref{SEC:Viewpoints} derives both the exact CUR decomposition and the CUR approximation, which are the primary viewpoints given in the literature.  These methods are compared and shown to be equivalent in the exact decomposition case in Section \ref{SEC:Proof}, and their difference is illustrated in the approximation case (Section \ref{SEC:Difference}).    We also give a history of the decomposition and its influences in Section \ref{SUBSEC:History}.

%%%%%%%%%%%%%%%%%%%%%%%%%%%%%%%%%%%%%%%%%%%%%%%%
%%%%%%%%%%%%%%%%%%%%%%%%%%%%%%%%%%%%%%%%%%%%%%%%
%%%%%%%%%%%%   Preliminaries          %%%%%%%%%%
%%%%%%%%%%%%%%%%%%%%%%%%%%%%%%%%%%%%%%%%%%%%%%%%
%%%%%%%%%%%%%%%%%%%%%%%%%%%%%%%%%%%%%%%%%%%%%%%%
\section{Notations}\label{SEC:Notation}

Here, the symbol $\K$ will represent either the real or complex field, and we will denote by $[n]$, the set of integers $\{1,\dots,n\}$.  In the sequel, we will often have occasion to speak of submatrices of a matrix $A\in\K^{m\times n}$ with respect to certain columns and rows.  For this, we will use the notation $A(I,:)$ to denote the $|I|\times n$ row submatrix of $A$ consisting only of those rows of $A$ indexed by $I\subset[m]$, and likewise $A(:,J)$ will denote the $m\times|J|$ column submatrix of $A$ consisting only of those columns of $A$ indexed by $J\subset[n]$.  Therefore, $A(I,J)$ will be the $|I|\times|J|$ submatrix of those entries $a_{ij}$ of $A$ for which $(i,j)\in I\times J$.  We will also allow columns and rows to be repeated, but we still use the same notation in this case.  Here, $C$ will always denote $A(:,J)$ for some $J\subset[n]$ (possibly with repeated entries), $R$ to denote $A(I,:)$ for $I\subset[m]$, and $U=A(I,J)$.

The Singular Value Decomposition (SVD) of a matrix $A$ will typically be denoted by $A=W_A\Sigma_AV_A^*$ with the use of $W$ being preferred to the typical usage of $U$ since the latter will stand for the middle matrix in the CUR decomposition.  The truncated SVD of order $k$ of a matrix $A$ will be denoted by $A_k=W_k\Sigma_kV_k^*$, where $W_k$ comprises the first $k$ left singular vectors, $\Sigma_k$ is a $k\times k$ matrix containing the largest $k$ singular values, and $V_k$ comprises the first $k$ right singular vectors.  We will always assume that the singular values are positioned in descending order, and label them $\sigma_1\geq\sigma_2\geq\dots\geq\sigma_r\geq0$. To specify the matrix involved, we may also write $\sigma_i(A)$ for the $i$--th singular value of $A$.

Given a matrix $A\in\K^{m\times n}$, its Moore--Penrose pseudoinverse will be denoted by $A^\dagger\in\K^{n\times m}$.  We recall for the reader that this pseudoinverse is unique and satisfies the following properties: (i) $AA^\dagger A = A$, (ii) $A^\dagger AA^\dagger = A^\dagger$, and (iii) $AA^\dagger$ and $A^\dagger A$ are Hermitian.  Additionally, given the SVD of $A$ as above we have a simple expression for its pseudoinverse as $A^\dagger=V_A\Sigma_A^\dagger W_A^*$, where $\Sigma^\dagger$ is the $n\times m$ matrix with diagonal entries $\frac{1}{\sigma_i(A)}$, $i=1,\dots,k=\rank(A)$. 

Our analysis will consider a variety of matrix norms.  Some of the most important are the spectral norm $\|A\|_2:=\sup\{\|Ax\|_2:x\in S_{\K^n}\}$, where $S_{\K^n}$ is the unit sphere of $\K^n$ (in the Euclidean norm).  It is a useful fact that $\|A\|_2 = \sigma_1(A)$.  The Frobenius norm is another common norm that will be used, and has an entrywise definition, but also can be represented using the singular values as well, to wit
\[ \|A\|_F := \left(\sum_{i,j}a_{i,j}^2\right)^\frac12 = \left(\sum_i \sigma_i(A)^2\right)^\frac12.\]  We also utilize the Schatten $p$--norms, $1\leq p\leq\infty$ (which are defined by the right-most term in the Frobenius norm expression above but with the $\ell_2$ norm of the singular values replaced with an $\ell_p$ norm). 

Finally, we will use $\mathcal{N}(A)$ and $\mathcal{R}(A)$ to denote the nullspace and range of $A$, respectively.%; the symbol $f\gtrsim g$ will mean that $f\geq cg$ for some universal constant $c$, and for vectors $x,y\in\K^n$, $x\otimes y = xy^*$.

%%%%%%%%%%%%%%%%%%%%%%%%%%%%%%%%%%%%%%%%%%%
%%%%%%%%%%%%%%%%%%%%%%%%%%%%%%%%%%%%%%%%%%%
%%%%%%    2 Viewpoints             %%%%%%%%
%%%%%%%%%%%%%%%%%%%%%%%%%%%%%%%%%%%%%%%%%%%
%%%%%%%%%%%%%%%%%%%%%%%%%%%%%%%%%%%%%%%%%%%
\section{Two viewpoints on CUR}\label{SEC:Viewpoints}

%There appear to be two distinct starting points in most of the literature revolving around the CUR decomposition.  The first is as an exact matrix \textit{decomposition}, or \textit{factorization}.   

%The second starting point is to find a \textit{low-rank approximation} to a given matrix.  In this modern era of large-scale, high-dimensional data analysis, dimension reduction techniques are absolutely crucial to leveraging data to make accurate conclusions about the world around us.  Despite significant advances in computational power, the data that we collect today is rarely amenable to fast computations without some sort of dimension reduction beforehand.  Moreover, data often has an intrinsically low-dimensional structure, which may be elucidated by dimension reduction techniques.

In this section, we will illustrate the conclusion one would obtain for a CUR decomposition or approximation based on each launching point mentioned here.  %Then we show that in a certain case (i.e. when an exact decomposition is obtained) these vantage points provide the exact same answer (the decomposition and approximation are one and the same).  Finally, we also discuss how the two conclusions are distinct from each other in the case one desires to use CUR as a low-rank approximation.  %For ease of reading, we will withhold citations and historical notes about the theorems presented in this section until Section \ref{SUBSEC:History}, and simply alert the reader here that most of the results are known in some fashion, but it is our aim to provide some context and comparison of them here.

%%%%%%%%%%%%%%%%%%%%%%%%%%%%%%%%%%%%%%%%%%%%%%%%
%%%%%%%%%%%%%%%%%%%%%%%%%%%%%%%%%%%%%%%%%%%%%%%%
%%%%%%%%%%   Exact Decomposition      %%%%%%%%%%
%%%%%%%%%%%%%%%%%%%%%%%%%%%%%%%%%%%%%%%%%%%%%%%%
%%%%%%%%%%%%%%%%%%%%%%%%%%%%%%%%%%%%%%%%%%%%%%%%
\subsection{Exact Decomposition: $A = CU^\dagger R$}

For our first tale regarding CUR, we begin by asking the question: given a matrix $A\in\K^{m\times n}$ with rank $k<\min\{m,n\}$, can we decompose it into terms involving only some of its columns and some of its rows?  Particularly, if we choose $k$ columns of $A$ which span the column space of $A$ and $k$ rows which span the row space of $A$, then we should be able to stitch these linear maps together to get $A$ itself back.   The answer, as it turns out, is yes as the following theorem shows (the intuition of the previous statement will be demonstrated in a subsequent figure).

\begin{theorem}\label{THM:CUR}
Let $A\in\K^{m\times n}$ have rank $k$, and let $I\subset[m]$ and $J\subset[n]$ with $|I|=t\geq k$ and $|J|=s\geq k$.  Let $C=A(:,J)$, $R=A(I,:)$, and $U=A(I,J)$.  If $\rank(U)=\rank(A)$, then %Let $C=A(:,J)$ be the $m\times s$ column submatrix of $A$ with respect to the index set $J$, and let $R = A(I,:)$ be the $t\times n$ row submatrix of $A$ with respect to the index set $J$.  Let $U=A(I,J)$ be the $t\times s$ submatrix of $A$ whose entries are $\{a_{i,j}\}_{(i,j)\in I\times J}$.  
\[ A = CU^\dagger R.\]
\end{theorem}

\begin{proof}
Given the constraint on $U$, it follows that $\rank(C)=\rank(R)=\rank(A)$.  Thus, there exists an $X\in\K^{s\times n}$ such that $A=CX$ (in fact there are infinitely many if $s>\rank(A)$).  Now suppose that $P_I$ is the row-selection matrix which picks out the rows of $A$ according to the index set $I$.  That is, we have $P_IA = R$, and likewise $P_IC = U$.  Then the following holds:
\begin{displaymath}
\begin{array}{lll}
A = CX & \Leftrightarrow & P_IA = P_ICX\\
& \Leftrightarrow & R = UX.\\
\end{array}
\end{displaymath}
The second equivalence is by definition of $U, R$, and $P_I$, while the forward direction of the first equivalence is obvious.  The backward direction holds by the assumption on the rank of $U$, and hence on $C$ and $R$.  That is, the row-selection matrix $P_I$ eliminates rows that are linearly dependent on the rest, and so any solution to $P_IA = P_ICX$ must be a solution to $A=CX$.  Finally, it suffices to show that $X = U^\dagger R$ is a solution to $R=UX$.  By the same argument as above, it suffices to show that $U^\dagger R$ is a solution to $RP_J=U = UXP_J$ if $P_J$ is a column-selection matrix which picks out columns according to the index set $J$.  Thus, noting that $UU^\dagger RP_J = UU^\dagger U = U$ completes the proof.
\end{proof}

\begin{figure}[ht]
\includegraphics[scale=0.4]{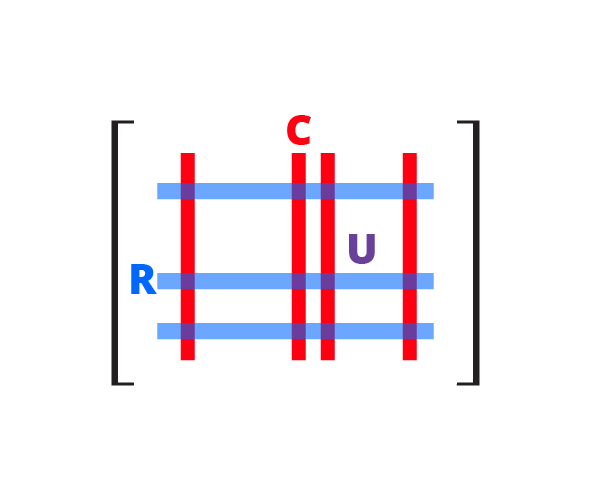}
\caption{Illustration of the CUR decomposition. $C$ is the red column submatrix of $A$, while $R$ is the blue row submatrix of $A$, and $U$ is their intersection (naturally purple).}\label{FIG:CUR}
\end{figure}

Figure \ref{FIG:CUR} provides an illustration of the CUR decomposition from a matrix point of view, whereas Figure \ref{FIG:CURSpace2} shows the intuition of the decomposition based on the viewpoint of the linear operators that the matrices represent.

\begin{remark}\label{REM:CUR}
Upon careful examination of the proof of Theorem \ref{THM:CUR}, we observe that the conclusion $A=CU^\dagger R$ holds also in the event that columns and rows of $A$ are repeated.  That is, if $I=\{i_1,\dots,i_t\}$, $i_k\in[n]$ and $J=\{j_1,\dots,j_s\}$, $j_k\in[m]$ (where the $i_k$ and $j_k$ are not necessarily distinct), and $C$ consists of columns $A(:,i_k)$, $k\in[t]$ and $R$ consists of rows $A(j_k,:)$, $k\in[s]$, with $U_{k,\ell}=A_{i_k,j_\ell}$, then provided $\rank(U)=\rank(A)$, we have $A=CU^\dagger R$.  %This observation will be used in Section \ref{SEC:ColumnSelection}.
\end{remark}
%\begin{figure}[h]
%		{\includegraphics[scale=0.7]{Diagram-1.png}}
%		\caption{A diagram of the CUR decomposition viewed as linear operators.}\label{FIG:CURSpace}
%	\end{figure}
\begin{figure}[ht]
		\includegraphics[scale=0.7]{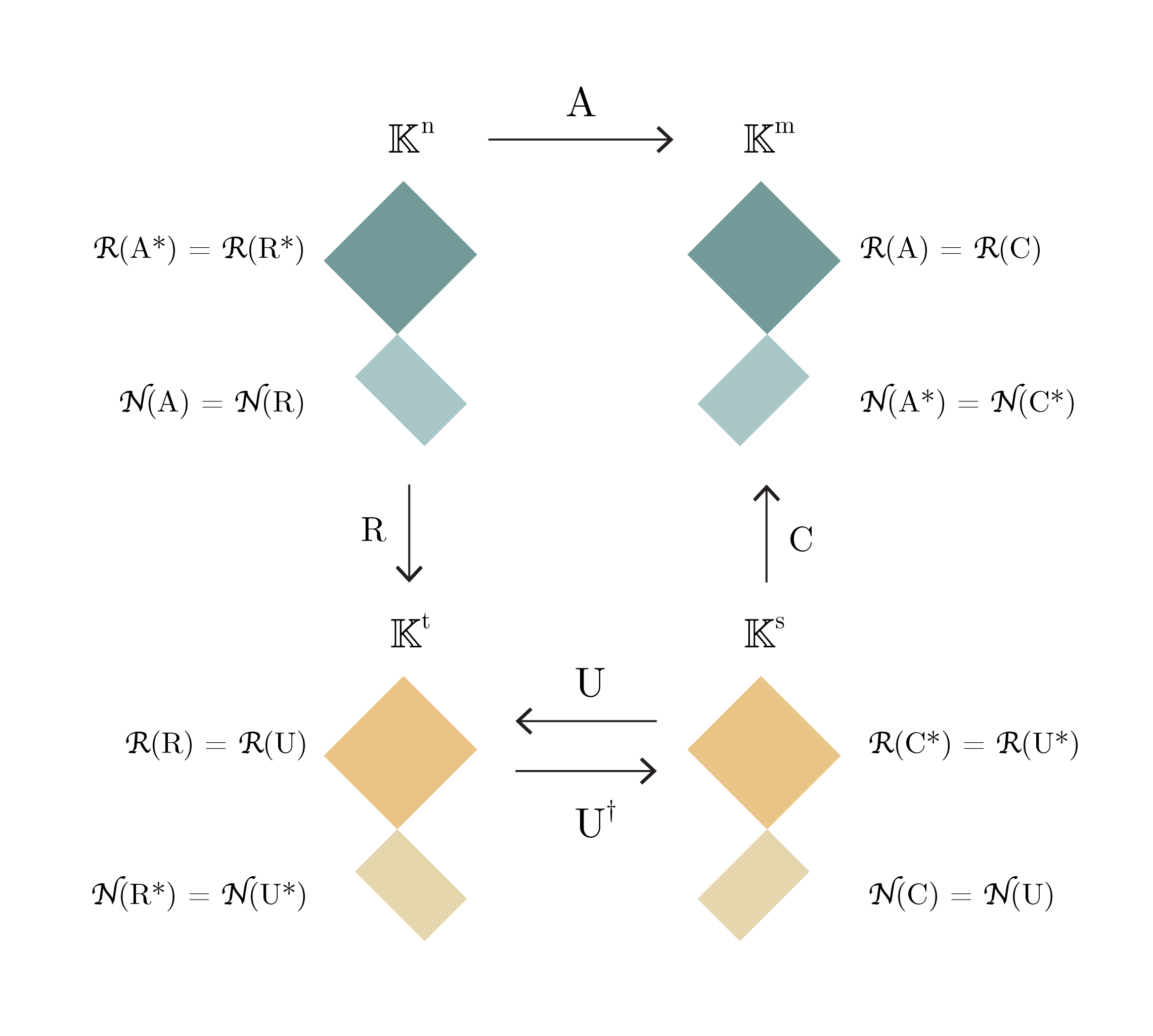}
		\caption{A diagram of the CUR decomposition viewed as linear operators. The illustration of the decomposition of the spaces is inspired by \cite{Strang}.  Because $R$ and $C$ capture the essential information of $A$, the null spaces correspond as shown (see also Lemma \ref{LEM:Projections}), and $U^\dagger$ is the inverse of $U$ on its range, which allows the diagram to commute.}\label{FIG:CURSpace2}
\end{figure}

%%%%%%%%%%%%%%%%%%%%%%%%%%%%%%%%%%%%%%%%%%%%%%%%
%%%%%%%%%%%%%%%%%%%%%%%%%%%%%%%%%%%%%%%%%%%%%%%%
%%%%%%%%    Low Rank Approximation    %%%%%%%%%%
%%%%%%%%%%%%%%%%%%%%%%%%%%%%%%%%%%%%%%%%%%%%%%%%
%%%%%%%%%%%%%%%%%%%%%%%%%%%%%%%%%%%%%%%%%%%%%%%%
\subsection{Best Low Rank Approximation: $A \approx CC^\dagger AR^\dagger R$}

A low rank \textit{approximation} of a matrix $A$ is a matrix with small (compared to the size of the matrix) rank which is ideally close to $A$ in norm.  Many low rank approximation (and indeed decomposition) methods begin with the idea that perhaps a basis other than the canonical one is a ``better" basis in which to represent a given matrix, where the term better is vague and dependent upon the context.  The second viewpoint on CUR stems from this same idea and is, in our opinion, the one more closely tied to those interested in data science, whether in theory or practice.  So with that in mind, suppose that we would like to take a matrix $A\in\K^{m\times n}$, and find a rank $k$ approximation to it given some fixed $k$.  Of course, it is well known that if one desires the \textit{best} rank $k$ approximation to $A$, then one needs look no further than its truncated SVD.  That is, if $W_k, \Sigma_k, V_k$ are the first $k$ left singular vectors, singular values, and right singular vectors, respectively, then one has
\begin{equation}\label{EQ:SVD Minimizer} \underset{X: \rank(X)=k}{\text{argmin}}\; \|A-X\|_{\xi} = W_k\Sigma_kV_k^*\end{equation} in the case that $\xi$ is any Schatten $p$--norm.

Given this observation, the reader might be forgiven for thinking, why should one look any further for a low-rank approximation for $A$ when the SVD provides the best?  Well, this is not the whole story, of course.  One reason to continue the search is that the SVD fails to preserve much of the structure of $A$ as a matrix.  For example, suppose that $A$ is a sparse matrix; then in general $U$ and $V$, hence $W_k$ and $V_k$, fail to be sparse.  Therefore, the SVD does not necessarily remain faithful to the matricial structure of $A$.  For yet another reason, we turn to the wisdom of Mahoney and Drineas \cite{DMPNAS}.  As they point out, a major factor in analyzing data is interpretability of the results.  Consequently, if one manipulates data in some fashion, one must do it in such a way as to still be able to make a meaningful conclusion.

Let us take the SVD as a case in point: suppose in a medical study, a researcher observes a large number $n$ of gene expression levels in $m$ patients and concatenates the data into an $m\times n$ matrix $A$. Supposing the desired outcome is to determine which genes are most indicative of cancer risk in patients, the researcher attempts to reduce the dimension of the data significantly, and so takes the truncated SVD of this matrix, and looks at the data in the $k$--dimensional basis $W_k$.  But what does a singular vector in $W_k$ correspond to?  It will generally be a linear combination of the genes; so what would it mean, say, that the first two singular vectors capture the majority of information in the data if the singular vectors are combinations of all of the gene expressions?  In using the SVD, interpretability of the data has been utterly lost.

Finally, computing the full SVD of a matrix $A$ is expensive (the na\"{i}ve direct algorithm requires $O(\min\{mn^2, nm^2\})$ operations).  Nonetheless, it is more economical to compute the truncated SVD; indeed, computing $A_k$ requires only $O(mnk)$ operations \cite{tropp}.

So what would be a better alternative?  It seems natural to ask: can we choose only a few representative columns of our data matrix $A$ such that they essentially capture all of the necessary information about $A$?  Or more precisely, can we project $A$ onto the space spanned by some representative columns such that the result is close to $A$ in norm?  Unsurprisingly, this problem is important enough to be named, and is typically called the \textit{Column Subset Selection Problem}.  Before stating the problem, consider the preliminary observation that if $C$ is a column submatrix of $A$, then the following holds for $\xi = 2$ or $F$ \cite[Theorem 10.B.7]{MOA_2011}:
\begin{equation}
\label{EQ:MinA-CX}\underset{X}{\argmin}\;\|A-CX\|_{\xi} = C^\dagger A.  \end{equation} 
This property follows from the fact that the Moore--Penrose pseudoinverse gives the least-squares solution to a system of equations.  With this observation in hand, combined with the knowledge that $CC^\dagger:\K^m\to\K^m$ is the projection operator onto the column space of $C$, the Column Subset Selection Problem may be stated as follows.

\begin{problem}[Column Subset Selection Problem]
Given a matrix $A\in\K^{m\times n}$ and a fixed $k\in[n]$, find a column submatrix $C=A(:,J)$ which solves the following:
\[ \underset{J\subset[n], |J|=k}{\underset{C=A(:,J)}{\min}}\; \|A-CC^\dagger A\|_\xi,\]
where $\xi$ is a norm allowed to be specified -- typically chosen to be either $2$ or F.
\end{problem}

The astute observer will notice that this problem is difficult in general (more on its complexity in Section \ref{SUBSEC:History}).  Indeed, there are $\binom{n}{k}$ choices of matrices $C$ over which to minimize.  However, let us set this difficulty aside for the moment and return to the problem at hand.  Equation \eqref{EQ:MinA-CX} tells us that given a column submatrix $C$ which solves the Column Selection Problem, $A$ is best represented by $A\approx CC^\dagger A$.  But why stop there?  We may as well also select some rows of $A$ which best capture the essential information of its row space.  Similar to \eqref{EQ:MinA-CX}, one may easily show that, given a row submatrix $R$ of $A$, the following holds for $\xi=2$ or $F$:
\[ \underset{X}\argmin\;\|A-XR\|_\xi = AR^\dagger.\]
Therefore, we now wish to find the best rows which minimize the argument above.  Rather than calling this the ``Row Subset Selection Problem," simply note that this is equivalent to solving the Column Subset Selection Problem on $A^*$.

It is now natural to stitch these tasks together, and attempt to find the minimizer of $\|A-CZR\|_\xi$ given a fixed $C$ and $R$.

\begin{proposition}[\cite{stewart_minimizer}]\label{PROP:CAR}
Let $A\in\K^{m\times n}$ and $C$ and $R$ be column and row submatrices of $A$, respectively.  Then the following holds:
\[ \underset{Z}{\text{argmin}}\;\|A-CZR\|_F = C^\dagger AR^\dagger.\]
\end{proposition}

Given the result of Proposition \ref{PROP:CAR}, much of the literature surrounding the CUR decomposition takes
\[ A\approx CC^\dagger AR^\dagger R\]
to be the CUR decomposition of $A$.  However, to be more precise, we will herein term this a CUR \textit{approximation} of $A$.

It should be noted that Proposition \ref{PROP:CAR} is not true for spectral norm as the follow example demonstrates.
\begin{example}
Consider \[A=\begin{bmatrix}
1& 1\\
1&2\\
\end{bmatrix}, C=\begin{bmatrix}
1\\
1\\
\end{bmatrix}, R=\begin{bmatrix}
1&1\\
\end{bmatrix}.\]  First note that in this case when evaluating $\text{argmin}_z\|A-CzR\|_\xi$, $z$ is simply a scalar. It is a simple exercise to demonstrate that $\|A-CzR\|_F^2 = 3(1-z)^2+(2-z)^2$, which is minimized whenever $z=\frac54 = C^\dagger AR^\dagger$. However, under the 2-norm, one can show that $z=\frac{5}{4}$ is not the optimal solution by considering $z=\frac{3}{2}$. %the maximal eigenvalue of $A-CzR$ and computing the minimizer explicitly.
 Thus Proposition \ref{PROP:CAR} does not hold for the spectral norm in general.
\end{example}
\begin{comment}
 Because $A$ is symmetric,  $\|A-zCR\|_2$ equals the maximum absolute value of the eigenvalues of $A-zCR$.
\begin{eqnarray*}
\arg\min_{z}\|A-zCR\|_2&=&\arg\min_{z}\left\|\begin{bmatrix}
1-z&1-z\\
1-z&2-z\\
\end{bmatrix}\right\|_2\\
&=&\arg\min_{z}\max\left\{  \frac{1+2(1-z)+\sqrt{1+4(1-z)^2}}{2},\left|\frac{1+2(1-z)-\sqrt{1+4(1-z)^2}}{2} \right| \right\}\\
&=&\arg\min_{z}\begin{cases}
 \frac{1+2(1-z)+\sqrt{1+4(1-z)^2}}{2}, 1-z\geq 0\\
 \frac{-1-2(1-z)+\sqrt{1+4(1-z)^2}}{2}, 1-z<0
\end{cases}=1.
\end{eqnarray*}
\end{comment}

\begin{example}
Another example of a different sort is to take $A = I$, and $C$ and $R$ to again be the first column and row, respectively.  Then $C^\dagger AR^\dagger=1$, but the eigenvalues of $A-CzR$ are $1$ and $1-z$.  So any $z\in[0,2]$ yields a minimum value for $\|A-CzR\|_2$ of 1.  This example also illustrates the key fact that there may be a continuum of matrices $Z$ for which $\|A-CZR\|_2$ is constant.
\end{example}

%%%%%%%%%%%%%%%%%%%%%%%%%%%%%%%%%%%%%%%%%%%%%%%%
%%%%%%%%%%%%%%%%%%%%%%%%%%%%%%%%%%%%%%%%%%%%%%%%
%%%%%%%%%%%%   Exact Case             %%%%%%%%%%
%%%%%%%%%%%%%%%%%%%%%%%%%%%%%%%%%%%%%%%%%%%%%%%%
%%%%%%%%%%%%%%%%%%%%%%%%%%%%%%%%%%%%%%%%%%%%%%%%
\section{Characterization of the Exact Decomposition}\label{SEC:Proof}

Previously, we discussed two starting points and conclusions for what is termed the CUR decomposition in the literature.  The purpose of this and the next section is to compare these two viewpoints.  First, we demonstrate in Theorem \ref{THM:Characterization} that in the exact decomposition case, these viewpoints are in fact one and the same. However, in the process of doing so, we do more by giving several equivalent characterizations of when an exact CUR decomposition is obtained. Before stating this theorem, we make note of some useful facts about the matrices involved.

\begin{lemma}\label{LEM:Projections}
Suppose that $A, C, U,$ and $R$ are as in Theorem \ref{THM:CUR}, with $\rank(A)=\rank(U)$.  Then $\mathcal{N}(C)=\mathcal{N}(U)$, $\mathcal{N}(R^*)=\mathcal{N}(U^*)$, $\mathcal{N}(A)=\mathcal{N}(R)$, and $\mathcal{N}(A^*)=\mathcal{N}(C^*)$.  Moreover,
\[ C^\dagger C = U^\dagger U,\qquad  RR^\dagger=UU^\dagger,\]
\[ AA^\dagger = CC^\dagger, \quad \text{and} \quad A^\dagger A = R^\dagger R.\]
\end{lemma}

\begin{proof}
As noted in the proof of Theorem \ref{THM:CUR}, the constraint that $\rank(U)=\rank(A)=k$ implies also that $C$ and $R$ have rank $k$ as well.  The statements about the null spaces follows directly from this observation.  To prove the moreover statements, notice that if $C\in\K^{m\times s}$, then $C^\dagger C:\K^s\to\K^s$ is the orthogonal projection onto $\mathcal{N}(C)^\perp$.  However, $\mathcal{N}(C)=\mathcal{N}(U)$ since their ranks are the same and $U$ is obtained by selecting certain rows of $C$, and hence $\mathcal{N}(C)^\perp=\mathcal{N}(U)^\perp$.  Therefore, $C^\dagger C$ is the orthogonal projection onto $\mathcal{N}(U)^\perp$, but this is $U^\dagger U$.  Similarly, $RR^\dagger$ is the projection onto $\mathcal{N}(R^*)^\perp = \mathcal{N}(U^*)^\perp$, whence $RR^\dagger=UU^\dagger$.  The final two statements follow by the same reasoning, so the details are omitted.
\end{proof}

\begin{lemma}\label{rankE}
	Let $A\in\mathbb{K}^{m\times n}$ and $B\in\mathbb{K}^{n\times p}$. If $\rank(A)=\rank(B)=n$, then $\rank(AB)=n$.
	\end{lemma}
	The proof of Lemma \ref{rankE} is a straightforward exercise using Sylvester's rank inequality, and so is omitted.
	
	\begin{corollary}\label{CR_RANK_U}
	   Let $A\in\mathbb{K}^{m\times n}$ with $\rank(A)=k$. Let $C=A(:,J)$ and $R=A(I,:)$ with $\rank(C)=\rank(R)=k$. Then $\rank(U)=k$ with $U=A(I,J)$.
	\end{corollary}
	\begin{proof}
	Assume that $A$ has truncated SVD $A=W_k\Sigma_k V_k^*$. Then $C=W_k\Sigma_k (V_k(J,:))^*$. Since $\rank(C)=k$, we have $\rank(V_k(J,:))=k$. Similarly, we can conclude that $\rank(W_k(I,:))=k$. Note that $U=A(I,J)=W_k(I,:)\Sigma_k (V_k(J,:))^*$, hence by Lemma \ref{rankE}, we have $\rank(U)=k$.
	\end{proof}
	
As a side note, one can also demonstrate a different form for $U$ aside from simply intersection of $C$ and  $R$.

\begin{proposition}\label{PROP:U=RAC}
	Suppose that $A$, $C$, $U$, and $R$  are as in Theorem \ref{THM:CUR} (but without any assumption on the rank of $U$). Then \[U=RA^{\dagger}C.\]
\end{proposition}
\begin{proof}
Let the full singular value decomposition of $A$ be $A=W_A\Sigma_A V_A^*$. Then $C=A(:,J)=W_A\Sigma_AV_A^*(:,J)$, $R=A(I,:)=W_A(I,:)\Sigma_AV_A^*$, and $U=U_A(I,:)\Sigma_AV_A^*(:,J)$.
Therefore, we have
\begin{eqnarray*}
	RA^{\dagger}C&=&W_A(I,:)\Sigma_AV_A^*V_A\Sigma_A^\dagger W_A^*W_A\Sigma_AV_A^*(:,J)\\
	&=&W_A(I,:)\Sigma_A\Sigma_A^\dagger\Sigma_AV_A^*(:,J)\\
	&=&W_A(I,:)\Sigma_AV_A^*(:,J)=U.
\end{eqnarray*}
\end{proof}	
	
We are now in a position to prove our main theorem characterizing exact CUR decompositions (restated here for clarity).

\begin{theorem}\label{THM:Characterization}
Let $A\in\K^{m\times n}$ and $I\subset[m]$, $J\subset[n]$ (possibly having redundant entries).  Let $C=A(:,J)$, $U=A(I,J)$, and $R=A(I,:)$.  Then the following are equivalent:
\begin{enumerate}[(i)]
    \item\label{ITEM:Rank} $\rank(U)=\rank(A)$
    \item\label{ITEM:CUR} $A=CU^\dagger R$
    \item\label{ITEM:ACCARR} $A = CC^\dagger AR^\dagger R$
    \item\label{ITEM:Adagger} $A^\dagger = R^\dagger UC^\dagger$
    \item\label{ITEM:Spans} $\rank(C)=\rank(R)=\rank(A)$.
\end{enumerate}
\end{theorem}
\begin{proof}
Theorem \ref{THM:CUR} and Remark \ref{REM:CUR} give the implication $(\ref{ITEM:Rank})\Rightarrow(\ref{ITEM:CUR})$.  To see $(\ref{ITEM:CUR})\Rightarrow(\ref{ITEM:Rank})$,  suppose to the contrary that $\rank(U)\neq\rank(A)$.  By construction of $U$, this implies that $\rank(U)<\rank(A)$.  On the other hand,
\[\rank(A)=\rank(CU^\dagger R)\leq\rank(U^\dagger)=\rank(U)<\rank(A), \]
which yields a contradiction.  

The forward direction of $(\ref{ITEM:Rank})\Leftrightarrow(\ref{ITEM:Spans})$ is easily seen, while the reverse direction is the content of Corollary \ref{CR_RANK_U}, and $(\ref{ITEM:Spans})\Rightarrow(\ref{ITEM:ACCARR})$ is obvious given that under the assumption on the spans, $AA^\dagger = CC^\dagger$ according to Lemma \ref{LEM:Projections}.  To see $(\ref{ITEM:ACCARR})\Rightarrow(\ref{ITEM:Spans})$, recall that by construction, $\sspan(C)\subset\sspan(A)$, but $(\ref{ITEM:ACCARR})$ implies that $\sspan(A)\subset\sspan(C)$.  A similar argument shows that $\sspan(R^*)=\sspan(A^*)$, which yields $(\ref{ITEM:Spans})$.

Now suppose that $(\ref{ITEM:Rank})$ (and equivalently ($\ref{ITEM:CUR}$)) holds.  By Proposition \ref{PROP:U=RAC}, $U=RA^\dagger C$, hence the following holds:
\begin{align*}
    R^\dagger UC^\dagger &= R^\dagger RA^\dagger CC^\dagger\\
    &= A^\dagger AA^\dagger AA^\dagger\\
    & = A^\dagger,
\end{align*}
where the second equality follows from Lemma \ref{LEM:Projections} (which requires the assumption $(\ref{ITEM:Rank})$) and the last from properties defining the Moore--Penrose pseudoinverse.  Thus $(\ref{ITEM:Rank})\Rightarrow(\ref{ITEM:Adagger})$. 
Conversely, suppose that $A^\dagger = R^\dagger U C^\dagger$.  Then by Proposition \ref{PROP:U=RAC},
\[ A = AA^\dagger A = AR^\dagger UC^\dagger A = AR^\dagger RA^\dagger CC^\dagger A.\]
Hence $\rank(A)=\rank(AR^\dagger RA^\dagger CC^\dagger A)\leq\rank(C)\leq\rank(A)$. Thus, $\rank(A)=\rank(C)$. Similarly, $\rank(A)=\rank(R)$, and an appeal to Corollary \ref{CR_RANK_U} completes the proof of $(\ref{ITEM:Adagger})\Rightarrow(\ref{ITEM:Rank})$.
\end{proof}

Theorem \ref{THM:Characterization} provides several novel characterizations of exact CUR decompositions, and in particular, the equivalence of conditions (\ref{ITEM:CUR}) and $(\ref{ITEM:ACCARR})$ demonstrates that the two proposed viewpoints indeed match in the exact decomposition case (and only in this case).

We now turn to some auxiliary observations.

\begin{proposition}\label{PROP:Udagger}
Suppose that $A, C, U,$ and $R$ satisfy any of the equivalent conditions of Theorem \ref{THM:Characterization}.  Then \[ U^\dagger=C^\dagger A R^\dagger.\]
\end{proposition}

\begin{proof}
On account of Lemma \ref{LEM:Projections} and the fact that $A=CU^\dagger R$,
\begin{align*}
C^\dagger AR^\dagger & = C^\dagger CU^\dagger  R R^\dagger \\
& = U^\dagger UU^\dagger UU^\dagger\\
& = U^\dagger,
\end{align*}
where the final step follows from basic properties of the Moore--Penrose pseudoinverse.
\end{proof}

\begin{remark}
The condition in Proposition \ref{PROP:Udagger} is not sufficient; if $U^\dagger=C^\dagger A R^\dagger$, then $\rank(A)$ may not equal to $\rank(U)$. For example, let
\[A=\begin{bmatrix}
0&I\\
I&0
\end{bmatrix},\] and let $C=\begin{bmatrix}
0\\
I
\end{bmatrix}$ and $R=\begin{bmatrix}
0&I
\end{bmatrix}$. Then $U=U^\dagger=0$, and $C^\dagger A R^\dagger=\begin{bmatrix}
0 &I
\end{bmatrix}\cdot \begin{bmatrix}
0&I\\
I&0
\end{bmatrix}\cdot \begin{bmatrix}
0\\
I
\end{bmatrix}=0$. Thus $U^\dagger=C^\dagger AR^\dagger$, but $\rank(A)\neq \rank(U)$.
\end{remark}

Note that Proposition \ref{PROP:U=RAC} holds for any choice of column and row submatrices $C$ and $R$.  However, Proposition \ref{PROP:Udagger} requires the additional assumption that $U$ and $A$ have the same rank.  Recall that Proposition \ref{PROP:Udagger} does not follow immediately from Proposition \ref{PROP:U=RAC} without this additional assumption given the fact that $(AB)^\dagger$ is not $B^\dagger A^\dagger$ in general.  Additionally, since the conclusion of Proposition \ref{PROP:Udagger} does not imply $\rank(U)=\rank(A)$, then it cannot imply any of the equivalent conditions given in Theorem \ref{THM:Characterization} in general.  We end this section with an interesting question related to this proposition.  Evidently, if $U^\dagger=C^\dagger AR^\dagger$, then $CU^\dagger R = CC^\dagger AR^\dagger R$.  Does the converse always hold? Note that if $\rank(U)=\rank(A)$, then the converse is true by Theorem \ref{THM:Characterization} because both quantities are $A$.  However, to show that it holds in general, one must determine if it holds in the case that $A\neq CU^\dagger R$.  At the moment, we leave this as an open question; however, numerical experiments give evidence that it is possibly true.

\begin{comment}
Oddly enough, if $A=CU^\dagger R$, then one can show that its pseudoinverse is given by $R^\dagger U C^\dagger$, which is not expected in general.
\begin{proposition}\label{PROP:Adagger}
Suppose that $A,C,U,$ and $R$ are as in Theorem \ref{THM:CUR} with $\rank(U)=\rank(A)$.  Then 
\[ A^\dagger = R^\dagger U C^\dagger. \]
\end{proposition}
\begin{proof}
By Proposition \ref{PROP:U=RAC}, $U=RA^\dagger C$, hence the following holds:
\begin{align*}
    R^\dagger UC^\dagger &= R^\dagger RA^\dagger CC^\dagger\\
    &= A^\dagger AA^\dagger AA^\dagger\\
    & = A^\dagger,
\end{align*}
where the second equality follows from Lemma \ref{LEM:Projections} and the last from properties defining the Moore--Penrose pseudoinverse.
\end{proof}
\end{comment}

%%%%%%%%%%%%%%%%%%%%%%%%%%%%%%%%%%%%%%%%%%%%%%%%
%%%%%%%%%%%%%%%%%%%%%%%%%%%%%%%%%%%%%%%%%%%%%%%%
%%%%%%%%%%%%   Distinctness           %%%%%%%%%%
%%%%%%%%%%%%%%%%%%%%%%%%%%%%%%%%%%%%%%%%%%%%%%%%
%%%%%%%%%%%%%%%%%%%%%%%%%%%%%%%%%%%%%%%%%%%%%%%%
\section{Distinctness in the Approximation Case}\label{SEC:Difference}

Now let us give a simple example to show that in the CUR approximation case, the choice of $C^\dagger AR^\dagger$ for the middle matrix in the CUR approximation indeed gives a better approximation than using $U^\dagger$, and in fact these matrices are not the same in this case.

\begin{example}
Consider the full rank matrix \[
A = \begin{bmatrix}1 & 2\\ 3 & 4\\
\end{bmatrix},
\]
and let \[C = A(:,1) = \begin{bmatrix}1\\3\\
\end{bmatrix}, \text{ and } R = A(1,:) = \begin{bmatrix}1 & 2\\ \end{bmatrix}.\]

Then $U=U^\dagger = 1$, and notice that $\|A-CU^\dagger R\|_F = 2$.  However, minimizing the function $f(z) = \|A-CzR\|_F$ yields a minimal value of $1.0583$ at $z=0.76$, which is $C^\dagger AR^\dagger$.  This simple example demonstrates that in the approximation case when fewer rows and columns are chosen than the rank of $A$, the matrices $U^\dagger$ and $C^\dagger AR^\dagger$ are different in general.
\end{example}

\begin{comment}
\subsection{Storage and Complexity Considerations}

It is pertinent to discuss the storage cost and complexity of the approximations described above. As these are easily computed, we tabulate them here for the reader without proof for aesthetic purposes.  We assume here that the matrix rank is $k$, and that $r\geq k$ rows and $c\geq k$ columns are chosen as in Theorem \ref{THM:CUR}; all complexity and storage values in Table \ref{TAB:ComplexityStorage} are to be taken as $O(\cdot)$.

\begin{table}[h!]
\begin{tabular}{|c||c|c|}\hline
     Method & Complexity & Storage  \\ \hline\hline
     Full SVD & $\min\{m^2n,mn^2\}$ & $m^2+n^2+k$\\
     Truncated SVD & $mnk$ & $k(m+n+1)$\\
     $CU^\dagger R$ & $\min\{cr^2,c^2r\}$ & $mc+nr+k^2$ \\ 
    $CC^\dagger ARR^\dagger$ & $mc^2 + nr^2$  & $mc+nr+mn$    \\ \hline
\end{tabular}\caption{Complexity and storage sizes of different approximations.}\label{TAB:ComplexityStorage}
\end{table}
Note that in the case $c=r=k$, the storage of the truncated SVD and CUR decomposition are essentially the same, differing by a factor of $k(k-1)$, and the latter's complexity becomes $O(k^3)$, which is smaller than even the truncated SVD.
\end{comment}

%%%%%%%%%%%%%%%%%%%%%%%%%%%%%%%%%%%%%%%%%%%%%%%%
%%%%%%%%%%%%%%%%%%%%%%%%%%%%%%%%%%%%%%%%%%%%%%%%
%%%%%%%%%%%%   History                %%%%%%%%%%
%%%%%%%%%%%%%%%%%%%%%%%%%%%%%%%%%%%%%%%%%%%%%%%%
%%%%%%%%%%%%%%%%%%%%%%%%%%%%%%%%%%%%%%%%%%%%%%%%

\section{A History}\label{SUBSEC:History}

A precise history of the CUR decomposition (Theorem \ref{THM:CUR}) is somewhat elusive and its origins appear to be folklore at this point.  Many papers cite Gantmacher's book \cite{Gantmacher} without providing a specific location, but noting that the term \textit{matrix skeleton} is used therein.  The authors could not verify this source, as a search of the term \textit{skeleton} in a digital copy of the book turned up no results.  However, we find it implicitly in a paper by Penrose from 1956 \cite{Penrose56} (this is a follow-up paper to the one defining the pseudoinverse that now bears his name).  Therein, Penrose notes (albeit without proof) that any matrix may be written (possibly after rearrangement of rows and columns) in the form
\[ A = \begin{bmatrix} 
B & D\\ E & EB^{-1}D
\end{bmatrix},\]
where $B$ is any nonsingular submatrix of $A$ with $\rank(B)=\rank(A)$, whereupon one can immediately get a valid CUR decomposition for this matrix by choosing $R = \begin{bmatrix}B & D\end{bmatrix}$ and $C =\begin{bmatrix} B \\ E\end{bmatrix}$.  Subsequently, Theorem \ref{THM:CUR} is stated without proof in the case that $U$ is square and invertible in \cite{Goreinov}.  The proof for square submatrices $U$ that are not full rank appears in \cite{CaiafaTensorCUR}.  To the authors' knowledge, the first time the direct proof of the general rectangular $U$ case appears in the literature was in \cite{AHKS}; the proof given here is essentially the one given therein.

The trail of the CUR decomposition as a computational tool runs cold for some time after Penrose's paper, but may be picked up again in the works of Goreinov, Zamarashkin, and Tyrtyshnikov \cite{Goreinov,Goreinov2,Goreinov3}.  The authors therein take for granted the exact CUR decomposition of the form $A=CU^{-1}R$, which is a special case of Theorem \ref{THM:CUR} whenever exactly $\rank(A)$ rows and columns are chosen to form $U$.  From this launching point, they ask the question: if $A$ is approximately low rank, then how can one obtain a good CUR approximation to $A$ in the spectral norm?  However, their analysis is for general matrices $U$ rather than simply being of the form $U=A(I,J)$.  They provide precise estimates on CUR approximations in terms of a related min-max quantity.  These estimates are universal in the sense that the derived upper bound is not dependent upon the given matrix $A$.

The works of Goreinov, Zamarashkin, and Tyrtyshnikov are perhaps the modern starting point of CUR approximations, and have since sparked a significant amount of activity in the area.  Specifically, Drineas, Kannan, and Mahoney, have considered a large variety of CUR approximations inspired by the analysis of \cite{Goreinov}.  They again admit flexibility in the choice of the matrix $U$, and prove many relative and additive error bounds for their approximations, as well as determining algorithms for computing the approximations which are computationally cheap,  \cite{Bien,DKMIII,DM05,DMM08,DMPNAS}.  A typical result in these works quantifies how well a CUR approximation with randomly oversampled columns and rows approximates the truncated SVD up to some penalty.  Similar work has been done by Chiu and Demanet \cite{DemanetWu} for uniform sampling of columns and rows with an additional coherency assumption on the given matrix $A$.  For fast deterministic CUR methods, consult \cite{SorensenDEIMCUR,VoroninMartinsson}

Applications of CUR approximations have become prevalent, including works on astronomical object detection \cite{Yip14}, mass spectrometry imaging \cite{Yang15}, matrix completion \cite{Xu15}, the joint learning problem \cite{CURTPAMI}, and subspace clustering \cite{AHKS,aldroubi2018similarity}.

A very general framework for randomized low-rank matrix approximations is given in the excellent work of Halko, Martinsson, and Tropp \cite{tropp}, wherein they discuss CUR approximations as well as related methods such as interpolative decompositions and randomized approximations to the SVD.

The Column Subset Selection Problem (CSSP) has been well-studied in the theoretical computer science literature \cite{AltschulerGreedyCSSP, Boutsidis2009, Deshpande2010, LiDeterministicCSSP, OrdozgoitiCSSP, tropp2009column, Yang}.  Indeed as a dimension reduction tool for data analysis, the CSSP is completely natural.  Many such methods attempt to represent given data in terms of a basis of reduced dimension which capture the essential information of the data.  Whereas Principal Component Analysis may result in a loss of interpretability as mentioned before, column selection corresponds to choosing actual columns of the data, and hence the minimization in the CSSP is attempting to find the best features that capture the most information of the data.  Feature selection as a preconditioner to task-based machine learning algorithms -- e.g. neural network or support vector machine classifiers -- is a critical step in many applications, and thus a thorough understanding of the CSSP is important for the analysis of data.  Column selection has also been applied in drawing large graphs \cite{KhouryScheidegger}.

As for complexity, the CSSP is believed to be NP--hard, with a purported proof given by Shitov \cite{Shitov}; a proof of its UG--hardness was given already by \c{C}ivril \cite{Civril}.  UG--hardness is a relaxed notion which states that a problem is NP--hard assuming the \textit{Unique Games Conjecture} (see \cite{Khot} for a formal description).

%%%%%%%%%%%%%%%%%%%%%%%%%%%%%%%%%%%%%%%%%%%%%%%%
%%%%%%%%%%%%%%%%%%%%%%%%%%%%%%%%%%%%%%%%%%%%%%%%
%%%%%%%%%%%%   Acknowledgements       %%%%%%%%%%
%%%%%%%%%%%%%%%%%%%%%%%%%%%%%%%%%%%%%%%%%%%%%%%%
%%%%%%%%%%%%%%%%%%%%%%%%%%%%%%%%%%%%%%%%%%%%%%%%

\section*{Acknowledgements}  K. H. is partially supported by the National Science Foundation TRIPODS program, grant number NSF CCF--1740858. LX.H. is partially supported by NSF Grant DMS-1322099.  Much of the work for this article was done while LX.H. was a graduate student at Vanderbilt University.

We are indebted to Amy Hamm Design for making our sketches of Figures \ref{FIG:CUR}, and \ref{FIG:CURSpace2} a reality.  K. H. thanks Vahan Huroyan for many helpful discussions related to this work, and David Glickenstein and Jean-Luc Bouchot for comments on a preliminary version of the manuscript.

%%%%%%%%%%%%%%%%%%%%%%%%%%%%%%%%%%%%%%%%%%%%%%%%
%%%%%%%%%%%%%%%%%%%%%%%%%%%%%%%%%%%%%%%%%%%%%%%%
%%%%%%%%%%%%   Bibliography           %%%%%%%%%%
%%%%%%%%%%%%%%%%%%%%%%%%%%%%%%%%%%%%%%%%%%%%%%%%
%%%%%%%%%%%%%%%%%%%%%%%%%%%%%%%%%%%%%%%%%%%%%%%%
\bibliographystyle{plain}
\bibliography{HammHuang}

\end{document}